\documentclass[11pt]{amsart}
 \usepackage{amscd,amsmath,amsthm,amssymb}
 \usepackage{pstcol,pst-plot,pst-3d}
 \psset{unit=0.7cm,linewidth=0.8pt,arrowsize=2.5pt 4}

\newpsstyle{fatline}{linewidth=1.5pt}
\newpsstyle{fyp}{fillstyle=solid,fillcolor=verylight}
\definecolor{verylight}{gray}{0.97}
\definecolor{light}{gray}{0.93}
\definecolor{medium}{gray}{0.82}
 \def\F{{\mathcal F}}

 \def\opn#1#2{\def#1{\operatorname{#2}}} 
 \opn\chara{char} \opn\length{\ell} \opn\pd{pd} \opn\rk{rk}
 \opn\projdim{proj\,dim} \opn\injdim{inj\,dim} \opn\rank{rank}
 \opn\depth{depth} \opn\grade{grade} \opn\height{height}
 \opn\embdim{emb\,dim} \opn\codim{codim}
 \opn\hreg{hreg}

 \opn\Tr{Tr} \opn\bigrank{big\,rank}
 \opn\superheight{superheight}\opn\lcm{lcm}
 \opn\trdeg{tr\,deg}
 \opn\reg{reg} \opn\lreg{lreg} \opn\ini{in} \opn\lpd{lpd}
 \opn\size{size} \opn\sdepth{sdepth}
 \opn\link{link}\opn\fdepth{fdepth}\opn\lex{lex}
 \opn\div{div} \opn\Div{Div} \opn\cl{cl} \opn\Cl{Cl}
 \opn\Spec{Spec} \opn\Supp{Supp} \opn\supp{supp} \opn\Sing{Sing}
 \opn\Ass{Ass} \opn\Min{Min}\opn\Mon{Mon}
 \opn\Ann{Ann} \opn\Rad{Rad} \opn\Soc{Soc}
 \opn\Im{Im} \opn\Ker{Ker} \opn\Coker{Coker} \opn\Am{Am}
 \opn\Hom{Hom} \opn\Tor{Tor} \opn\Ext{Ext} \opn\End{End}
 \opn\Aut{Aut} \opn\id{id}
 
 \opn\nat{nat}
 \opn\pff{pf}
 \opn\d{d} \opn\GL{GL} \opn\SL{SL} \opn\mod{mod} \opn\ord{ord}
 \opn\Gin{Gin} \opn\Hilb{Hilb}\opn\sort{sort}
 \opn\aff{aff} \opn
 \con{conv} \opn\relint{relint} \opn\st{st}
 \opn\lk{lk} \opn\cn{cn} \opn\core{core} \opn\vol{vol}
 \opn\link{link} \opn\star{star}\opn\lex{lex}\opn\set{set}
 \opn\diam{diam}
 \opn\gr{gr}
 
 \def\pot#1#2{#1[\kern-0.28ex[#2]\kern-0.28ex]}

 \opn\dirlim{\underrightarrow{\lim}}
 \opn\inivlim{\underleftarrow{\lim}}
 \let\iso=\cong

 \def\Implies{\ifmmode\Longrightarrow \else
         \unskip${}\Longrightarrow{}$\ignorespaces\fi}
 \def\implies{\ifmmode\Rightarrow \else
         \unskip${}\Rightarrow{}$\ignorespaces\fi}
 \def\iff{\ifmmode\Longleftrightarrow \else
         \unskip${}\Longleftrightarrow{}$\ignorespaces\fi}

 \let\:=\colon
 \newtheorem{Theorem}{Theorem}[section]
 \newtheorem{Lemma}[Theorem]{Lemma}
 \newtheorem{Corollary}[Theorem]{Corollary}

 \newtheorem{Example}[Theorem]{Example}
 
 \newtheorem{Definition}[Theorem]{Definition}

 \let\epsilon\varepsilon
 \let\kappa=\varkappa
 \def\qed{\ifhmode\textqed\fi
       \ifmmode\ifinner\quad\qedsymbol\else\dispqed\fi\fi}
 \def\textqed{\unskip\nobreak\penalty50
        \hskip2em\hbox{}\nobreak\hfil\qedsymbol
        \parfillskip=0pt \finalhyphendemerits=0}
 \def\dispqed{\rlap{\qquad\qedsymbol}}
 
 \opn\dis{dis}
 \def\pnt{{\raise0.5mm\hbox{\large\bf.}}}
 
 \opn\Lex{Lex}

\begin{document}

 \title {The cleanness of (symbolic) powers of Stanley-Reisner ideals}

 \author {S.  Bandari and A.  Soleyman Jahan}

\address{Somayeh Bandari,  School of Mathematics, Institute for
Research in Fundamental Sciences (IPM) P. O. Box: 19395-5746,
Tehran, Iran.} \email{somayeh.bandari@yahoo.com}

\address{Ali Soleyman Jahan, Department of Mathematics, University of Kurdistan, Post Code 66177-15175, Sanandaj, Iran.}
\email{solymanjahan@gmail.com}

\keywords{Clean; Cohen-Macaulay; complete intersection; matroid;
symbolic power.\\
The research of the first author is supported by a grant from IPM
(No. 93130020)}

 \begin{abstract} Let $\Delta$ be a pure simplicial complex and $I_\Delta$ its Stanley-Reisner ideal in a polynomial ring $S$.
 We show that  $\Delta$ is a matroid (complete intersection)  if and only if
 $S/I_\Delta^{(m)}$ ($S/I_\Delta^m$) is clean for all $m\in\mathbb{N}$.
 If  $\dim(\Delta)=1$, we also prove that $S/I_\Delta^{(2)}$ ($S/I_\Delta^2$) is
 clean if and only if $S/I_\Delta^{(2)}$ ($S/I_\Delta^2$) is
 Cohen-Macaulay.
 \end{abstract}

 \maketitle

  \section*{Introduction}  Let $\Delta$ be a simplicial complex on vertex set $[n]=\{1,\ldots,n\}$ and
  $S=k[x_1,\ldots,x_n]$
  be the polynomial ring in $n$ indeterminate   over a field $k$. The Stanley-Reisner ideal of $\Delta$, $I_\Delta$, is
  defined by $I_\Delta:=(\prod_{i\in F}x_i\;:\; F\not\in \Delta)$.

  There is a bijection between squarefree monomial ideals $I\subset (x_1,\ldots,x_n)^2$ and simplicial complexes.
  Cohen-Macaulayness (resp. Buchsbaumness, cleanness, generalized Cohen-Macaulayness) of these ideals have been studied
  by several authors (see \cite{BH}, \cite{HMT}, \cite{HP},  \cite{MT},  \cite{St1},\cite{TT}, \cite{V}).
  Minh and Trung in \cite{MT}
  and Varbaro in \cite{Va} independently proved that $\Delta$ is a matroid  if and only if $S/I_\Delta^{(m)}$ is Cohen-Macaulay for
  all $m\in\mathbb{N}$, where $I_\Delta^{(m)}$ denote the $m^{th}$-symbolic power of $I_\Delta$. Later on, Terai and Trung \cite{TT}
  showed that $\Delta$ is a matroid if and only if $S/I_\Delta^{(m)}$ is Cohen-Macaulay for some integer $m\geq 3$. The similar
  characterizations of being Buchsbaum and generalized Cohen-Macaulay were also studied by them.
 Minh and Trung in \cite{MT1} proved that for simplicial complex
 $\Delta$ with $\dim(\Delta)=1$, $I_\Delta^{(2)}$ is  Cohen-Macaulay if and only if $\diam(\Delta)\leq 2$,
 where $\diam(\Delta)$ denotes the diameter of $\Delta$. We pursue this line of research further.

  This paper is organized as follows: in Section 1, we collect  some preliminaries which will be needed later.
  In Section 2, we show that if $\Delta$ is a matroid, then $S/I_\Delta^{(m)}$ is clean for all $m\in\mathbb{N}$;
  see Theorem \ref{main}. Since $I_\Delta$ is unmixed, in particular, this shows that $S/I_\Delta^{(m)}$ is Cohen-Macaulay for
  all $m\in\mathbb{N}$. Therefore this result covers one direction of the result of Minh and Trung \cite{MT} and Varbaro  \cite{Va}.
  Our proof is combinatorial and much easier than that was given in \cite{MT}. As our first corollary, by using \cite[Theorem 3.6]{TT},
  we show that if $\Delta$ is pure and $I=I_\Delta\subset S$, then the following conditions are equivalent:
 \begin{enumerate}
 \item[(a)] $\Delta$ is a matroid.
 \item[(b)] $S/I^{(m)}$   is clean  for all integer $m>0$.
  \item[(c)] $S/I^{(m)}$   is clean  for some integer $m\geq 3$.
 \item[(d)]$S/I^{(m)}$   is Cohen-Macaulay  for some integer $m\geq 3$.
  \item[(e)]$S/I^{(m)}$   is Cohen-Macaulay  for all integer $m>0$.
  \end{enumerate}
 Our second corollary asserts that a pure simplicial complex $\Delta$ is complete intersection if and
 only if $S/I^m_\Delta$ is clean for all $m\in\mathbb{N}$ and if and only if $S/I^m_\Delta$ is clean for
 some integer $m\geq 3$.

Let $I\subset S$ be a monomial ideal such that $S/I$ is (pretty)
clean. It is natural to ask whether $S/I^m$ or $S/I^{(m)}$ is  again
(pretty) clean for all $m\in\mathbb{N}$? Example \ref{counter1}
shows that the answer is negative in general.

In Section 3, we show that if  $I\subset S$ is the Stanley-Reisner
ideal of  a pure simplicial complex $\Delta$ with  $\dim\Delta =1$,
then
 for an integer $m>1$, $S/I^{(m)}$ ($S/I^m$) is clean if and only if
 $S/I^{(m)}$ ($S/I^m$) is Cohen-Macaulay.

\section{Preliminary}

A simplicial complex $\Delta$ on vertex set $[n]=\{1,\ldots,n\}$ is a collection of subsets of $[n]$ with the
property that if $F\subset G$ and $G\in \Delta$, then $F\in\Delta$. An element of $\Delta$ is called a {\em
 face}, and the maximal faces of $\Delta$, under inclusion, are called {\em facets}. We denote by $\mathcal
F(\Delta)$ the set of facets of $\Delta$. When $\mathcal
F(\Delta)=\{F_1,\ldots,F_t\}$, we write $\Delta=\langle F_1,\ldots,
F_t\rangle$. For each $F\in\Delta$, we set $\dim F:=|F|-1$, and
$$\dim\Delta:=\max\{\dim F\;:\; F\in\mathcal F(\Delta)\},$$
which is called the dimension of $\Delta$. A simplicial complex
$\Delta$ is called {\em pure} if all facets of $\Delta$ have the
same dimension. According to Bj\"{o}rner and Wachs \cite{BW}, a
simplicial complex $\Delta$ is said to be {\em (non-pure) shellable}
if there exists an order $F_1,\ldots, F_t$ of the facets of $\Delta$
such that for each $2\leq i\leq t$, $\langle F_1,\ldots,
F_{i-1}\rangle \cap \langle F_i \rangle$ is a pure $(\dim
F_i-1)$-dimensional simplicial complex. Such an ordering of facets is called a
{\em shelling}.

Let $S=k[x_1,\ldots,x_n]$ be a polynomial ring in $n$ indeterminate
over a field $k$. The Stanley-Reisner ideal of $\Delta$ is denoted
by $I_\Delta$ and defined as $I_\Delta:=(\prod_{i\in F}  x_i\;:\;
F\not\in\Delta)$. The facet ideal of $\Delta$ is defined as
$I(\Delta):=(\prod_{i\in F}x_i\;:\; F\in\mathcal F(\Delta))$.

The complement of a face $F$ is $[n]\setminus F$ and it is denoted
by $F^c$. Also, the complement of a simplicial complex $\Delta =
\langle F_1,\ldots,F_r\rangle$ is $\Delta^c:= \langle
F_1^c,\ldots,F_r^c\rangle$. The {\em Alexander dual} of $\Delta$ is
given by $\Delta^\vee:= \{F^c \;:\; F\not\in \Delta\}$. It is known
that for a simplicial complex $\Delta$ one has
$I_{\Delta^\vee}=I(\Delta^c)$.

 \begin{Definition}
{\em A {\em matroid} $\Delta$ is
a simplicial complex with the property that for all faces $F$ and $G$ in $\Delta$ with $|F|<|G|$,
there exists $i\in G\setminus F$ such that $F\cup\{i\}\in\Delta$.}
 \end{Definition}

 The above definition implies that each matroid is pure.
 As a consequence of \cite[Theorem 12.2.4]{HH}, a matroid can be characterized by
the following exchange property: a pure simplicial complex $\Delta$ is a matroid if and only if for any two facets $F$ and
 $G$ of $\Delta$ with $F\neq G$, and for any
$i\in F\setminus G$, there exists $j\in G\setminus F$ such that $(F\setminus\{i\})\cup\{j\}\in\Delta$.
 A squarefree monomial ideal $I$ in $S$ is called {\em matroidal} if  $I=I(\Delta)$, where $\Delta$ is a matroid.
 On the other hand by \cite[Theorem 2.1.1]{O},
$\Delta$ is a matroid  if and only if $\Delta^c$ is a matroid.
Altogether, as  $I(\Delta^c)=I_{\Delta^\vee}$, we have that $\Delta$ is a matroid  if and only if $I_{\Delta^\vee}$ is matroidal.

A simplicial complex $\Delta$ is called {\em complete intersection} if $I_\Delta$ is a complete intersection monomial ideal.
It is well known that each complete intersection simplicial complex is matroid.

 If $F\subseteq [n]$, then we put $P_F:=(x_i : i\in F)$. We have
 $I_\Delta=\bigcap_{F\in\mathcal F(\Delta^c)} P_F$, hence for
 each $m\in\mathbb{N}$, the  $m^{th}$-symbolic power of $I_\Delta$ is the ideal
$$I_\Delta^{(m)}=\bigcap_{F\in\mathcal F(\Delta^c)} P_F^m.$$

An ideal $I\subset S$ is called {\em normally torsionfree} if
$\Ass(S/I^m)\subseteq \Ass(S/I)$ for all $m\in\mathbb{N}$. If $I$ is
a squarefree monomial ideal, then $I$ is normally torsionfree if and
only if $I^{(m)}=I^m$ for all $m$; see \cite[Theorem 1.4.6]{HH}.

Let $I\subset S$ be a monomial ideal. A chain of monomial ideals
$$\mathcal{F}: I=I_0\subset I_1\subset\cdots\subset I_r=S$$ is called a {\em prime filtration} of $S/I$ if for
each $i=1,\ldots,r,$ there exists a monomial prime ideal $\frak p_i$ of $S$ such that
$I_i /I_{i-1}\iso S/\frak p_i$. The set of prime ideals
$\frak p_1,\ldots,\frak p_r$ which define the cyclic quotients of $\mathcal{F}$ will be
denoted by $\Supp \mathcal{F}$. It is known (and easy to see) that
$$\Ass S/I\subseteq \Supp \mathcal{F}\subseteq \Supp S/I.$$ Let $\Min I$ denote the set of
minimal prime ideals of $\Supp S/I$. Dress \cite{D} called a prime
filtration $\mathcal{F}$ of $S/I$ {\em clean} if $\Supp
\mathcal{F}=\Min I$ and in \cite[Theorem on page 53]{D} proved that
a simplicial complex $\Delta$ is (non-pure) shellable if and only if
$K[\Delta]$ is a clean ring. Pretty clean filtrations were defined
as a generalization of clean filtrations by Herzog and Popescu
\cite{HP}. A prime filtration $\mathcal{F}$ is called {\em pretty
clean} if for all $i<j$ for which $\frak p_i\subseteq \frak p_j$, it
follows that $\frak p_i=\frak p_j$. If $\mathcal{F}$ is a pretty
clean filtration of $S/I$, then $\Supp \mathcal{F}=\Ass S/I$; see
\cite[Corollary 3.4]{HP}. $S/I$ is called {\em clean} (resp. {\em
pretty clean}) if it admits a clean (resp. pretty clean) filtration.
Obviously, cleanness implies pretty cleanness.

Let $I\subset S$ be a monomial ideal. Then $S/I$ is
{\em sequentially Cohen-Macaulay} if there exist a chain of monomial   ideals
$$I=I_0\subset I_1\subset I_2\subset \cdots\subset I_r=S$$
such that each quotient $I_i/I_{i-1}$ is Cohen-Macaulay and
$$\dim(I_1/I_0)<\dim(I_2/I_1)<\cdots<\dim(I_r/I_{r-1}).$$
Clearly, if $S/I$ is Cohen-Macaulay, then it is sequentially Cohen-Macaulay. Also, if $S/I$ is
pretty clean, then by \cite{HP} it is sequentially Cohen-Macaulay.

The monomial ideal $I$ has {\em linear quotients} if one can order the set of minimal monomial
generators of $I$, $G(I)=\{u_1,\ldots,u_m\}$, such that the colon ideal $(u_1,\ldots,u_{i-1}):u_i$
is generated by a subset of the variables for all $i=2,\ldots,m$.  This means for each $j<i$, there
exists a $k<i$ such that $u_k:u_i=x_t$ and $x_t|u_j:u_i$, where  $t\in [n]$ and $u_k:u_i=u_k/\gcd(u_k,u_i)$.
In the case $I$ is squarefree, it is enough to show that for each $j<i$, there exists a $k<i$ such
that $u_k:u_i=x_t$ and $x_t|u_j$  for some  $t\in [n]$.

Let $u=\prod_{i=1}^nx_i^{a_i}$ be a monomial in $S$. Then
$$u^p:=\textstyle\prod\limits_{i=1}^n\textstyle\prod\limits_{j=1}^{a_i}x_{i,j}
\in k[x_{1,1},\ldots, x_{1,a_1},\ldots, x_{n,1},\ldots, x_{n,a_n}]$$
is called {\em polarization} of
$u$. Let  $I$ be a monomial ideal of $S$ with the unique set of minimal monomial generators $G(I)=
\{u_1,\ldots, u_m\}$. Then the ideal $I^p:=(u_1^p,\ldots, u^p_m)$ of $$T:=k[x_{i,j}\;:\;i=1,
\ldots, n,  j=1,\ldots ,a_i]$$ is called {\em polarization} of $I$.

 \section{Matroids and complete intersection simplicial complexes}

We will characterize matroids (complete intersection simplicial complexes) $\Delta$ in term
of the cleanness of the symbolic (ordinary) powers of $I_\Delta$.

 \begin{Theorem} \label{main}
 Let $I\subset S$ be the Stanley-Reisner ideal of a matroid $\Delta$. Then   $S/I^{(m)}$ is clean for all $m\in\mathbb{N}$.
 \end{Theorem}

 \begin{proof}
 Let $I=I_\Delta=\bigcap_{i=1}^t P_{F_i}$ be the irredundant irreducible primary decomposition of $I$, where
  $\Delta^c= \langle F_1,\ldots,F_t\rangle$ and $r=|F_i|$ for all $i=1,\ldots, t$. Then $I^{(m)}=\bigcap_{i=1}^tP_{F_i}^m$.
  By \cite[Theorem 3.10]{S} it is enough to show that
 $T/(I^{(m)})^p$ is clean.

  One can see by \cite[Proposition 2.3(3)]{F} that $((I^{(m)})^p)^\vee=\sum_{i=1}^r((P_{F_i}^m)^p)^\vee$. If $F_i=\{s_1,\ldots,s_r\}$,
  then by \cite[Proposition 2.5(2)]{F}   $(P_{F_i}^m)^p $ has the following irredundant irreducible primary decomposition
 $$(P_{F_i}^m)^p=\bigcap_{\substack{
            1\le t_j \le m\\
            \sum t_j\leq m+r-1}}
     (x_{s_1,t_1},\ldots,x_{s_r,t_r})$$
It follows that the ideal $J:=((I^{(m)})^p)^\vee$ is generated by
the monomials
$$ x_{i_1,a_1}x_{i_2,a_2}\ldots x_{i_r,a_r}\quad
\text{with}\quad \{i_1,\ldots,i_r\}\in \F(\Delta^c),$$
 where
$a_j$ are positive integers satisfying  $1\leq a_j\leq m$ and
$\sum_{j=1}^ra_j\leq m+r-1$. For showing that $T/(I^{(m)})^p$ is
clean, it is enough to show that
 $J$ has linear quotients; see for example \cite[Lemma 2.1]{BDS}.

 Now, we order the variables in $T$ as follows:

$ x_{i,a}> x_{j, b} \Leftrightarrow (i,a)<(j,b)$, and $(i,a)<(j,b)$
if $a<b$, or $a=b$ and $i<j$. Then we show that $J$ has linear
quotients with respect to the reverse lexicographical order of its
generators induced from the above ordering. Indeed,  let
$u=x_{i_1,a_1}x_{i_2,a_2}\ldots x_{i_r,a_r}$ and
$v=x_{j_1,b_1}x_{j_2,b_2}\ldots x_{j_r,b_r}$ be two monomials in
$G(J)$ with
 $u>v$. We have to show that  there exists $w\in G(J)$ with $w>v$ such that $w:v=x_{i_l,a_l}$ and $x_{i_l,a_l}|u$.

 Since $u>v$, there exists  an integer $t$ such that $x_{i_t,a_t}> x_{j_t, b_t}$ and $x_{i_k,a_k}=x_{j_k, b_k}$ for
 all  $k>t$. In particular, we have $a_t<b_t$, or $a_t=b_t$ and $i_t<j_t$.
We first claim that there exists $1\leq l\leq t$ such that
$$x_{j_1}\ldots x_{j_{t-1}}x_{i_l}x_{j_{t+1}}\ldots x_{j_r} \in G(I_{\Delta^\vee})=G(I(\Delta^c)).$$
This is obvious, if $x_{j_t}|x_{i_1}x_{i_2}\cdots x_{i_t}$, and  if
$x_{j_t}\nmid x_{i_1}x_{i_2}\cdots x_{i_t}$, then, as $I^\vee$ is
matroidal, it follows that there exists $1\leq l\leq t$ such that
$x_{j_1}\cdots x_{j_{t-1}}x_{i_l}x_{j_{t+1}}\cdots x_{j_{r}}\in
G(I^\vee)$.   Here, we used the fact that $i_k=j_k$ for
$k=t+1,\ldots,r$. Then
\[w:=x_{j_1,b_1}x_{j_2,b_2}\cdots x_{j_{t-1},b_{t-1}}x_{i_l,a_l}x_{j_{t+1},b_{t+1}}\cdots x_{j_{r},b_r}\in G(J),
\]
because $a_l\leq b_t$. Moreover, we have  $w:v=x_{i_l,a_l}$ and $x_{i_l,a_l}|u$.

Next, we will show that $w>v$. If $x_{i_l,a_l}>x_{j_{t-1},b_{t-1}}$, then $w>v$ because $x_{j_{t-1},b_{t-1}}>x_{j_t,b_t}$. Otherwise,
one has $x_{i_l,a_l}<x_{j_{t-1},b_{t-1}}$.  We know that $ a_t<b_t$ or $a_t=b_t$ and $i_t<j_t$. Since $a_l\leq a_t$, if $a_t<b_t$,
then $w>v$. Now, assume that $a_t=b_t$ and $i_t<j_t$. Since  $a_l< a_t$ or   $a_l=a_t$, and $i_l<i_t<j_t$, one has
$x_{i_l,a_l}>x_{j_t,b_t}$ and $w>v$.
\end{proof}

We shall use the  following lemma.

\begin{Lemma}\label{seq}
Let $I\subset S$ be a  monomial ideal. Then $S/I$ is Cohen-Macaulay if and only if $S/I$ is sequentially Cohen-Macaulay and $I$
is unmixed.
\end{Lemma}

\begin{proof}
If $S/I$ is Cohen-Macaulay, then it is obvious that $S/I$ is sequentially Cohen-Macaulay and $I$ is unmixed.
Conversely, assume that $S/I$ is
sequentially Cohen-Macaulay and  $I$ is unmixed. Then there exists a chain of monomial  ideals
$$I=I_0\subset I_1\subset I_2\subset \cdots\subset I_r=S$$
such that each quotient $I_i/I_{i-1}$ is Cohen-Macaulay and
$$\dim(I_1/I_0)<\dim(I_2/I_1)<\cdots<\dim(I_r/I_{r-1}).$$
By \cite[Lemma 1.2]{HPV},  $\depth(S/I)=\dim(I_1/I_0)$. On the other
hand by \cite[Proposition 2.5]{HP},
$\Ass(S/I)=\bigcup_{i=1}^r\Ass(I_i/I_{i-1})$.   Since $I$ is
unmixed, it follows that $\dim(S/I)=\dim(I_i/I_{i-1})$ for all $i$.
Hence $\depth(S/I)=\dim(I_1/I_0)=\dim(S/I),$ and so $S/I$ is
Cohen-Macaulay.
\end{proof}

If we combine our results with \cite[Theorem 3.6]{TT}, we get the following characterization of matroids.

\begin{Corollary}\label{main2}
 Let $\Delta$ be a pure simplicial complex and $I=I_\Delta\subset S$. Then the following conditions are equivalent:
 \begin{enumerate}
 \item[(a)] $\Delta$ is  a matroid.
 \item[(b)] $S/I^{(m)}$   is clean  for all integer $m>0$.
  \item[(c)] $S/I^{(m)}$   is clean  for some integer $m\geq 3$.
 \item[(d)]$S/I^{(m)}$   is Cohen-Macaulay  for some integer $m\geq 3$.
  \item[(e)]$S/I^{(m)}$   is Cohen-Macaulay  for all integer $m>0$.
  \end{enumerate}
\end{Corollary}

\begin{proof}
In view of Theorem \ref{main}, (a) \implies (b) holds.
 The implications (a)$\Leftrightarrow$ (d) $\Leftrightarrow$ (e)  follow from \cite[Theorem 3.6]{TT}.
 The implication (b)\implies (c) is trivial.

(c) \implies (d) Suppose that for some  integer $m\geq 3$,
   $S/I^{(m)}$   is clean. Then by \cite[Corollary 4.3]{HP}, $S/I^{(m)}$  is sequentially Cohen-Macaulay. On the other hand,
   $I^{(m)}$ is an unmixed  monomial ideal for all $m$, because $I$ is unmixed and $\Ass(S/I^{(m)})=\Ass(S/I).$
 Hence by  Lemma \ref{seq}, $S/I^{(m)}$  is  Cohen-Macaulay.
\end{proof}

It is known \cite{AV} that a simplicial complex $\Delta$ is complete
intersection if and only if $S/I_\Delta^m$ is Cohen-Macaulay for all
$m\in\mathbb{N}$. Since for a complete intersection monomial ideal
$I_\Delta$ the symbolic powers coincide with its ordinary powers,
 we have:

\begin{Corollary} \label{c.i}
Let $\Delta$ be a pure simplicial complex and $I=I_\Delta\subset S$. Then the
following conditions are equivalent:
\begin{enumerate}
\item[(a)]  $\Delta$ is a complete intersection.
\item[(b)] $S/I^m$ is clean for all integer $m>0$.
\item[(c)] $S/I^m$ is clean for some integer $m\geq 3$.
\item[(d)]  $S/I^m$ is Cohen-Macaulay for some integer $m\geq 3$.
\item[(e)]  $S/I^m$ is Cohen-Macaulay for all integer $m>0$.
\end{enumerate}
\end{Corollary}

\begin{proof}
The equivalences (a)$\Leftrightarrow$(d)$\Leftrightarrow$(e) follow from \cite[Theorem 4.3]{TT}.
The implication $(b)\implies (c)$ is obvious.
The proof of $(c)\implies (d)$ is similar to that of  the same case in Corollary \ref{main2}.
 Note that, as $S/I^m$ is clean for some integer $m\geq 3$,
it follows that
$$\Ass(S/I^m)=\Min(I^m)=\Min(I)=\Ass(S/I).$$
It remains to show  $(a)\implies (b)$.  Since $I$ is complete
intersection, for any $m>0$, one has
$\Ass(S/I^m)=\Min(I^m)=\Min(I)$. Hence by the definition of symbolic
powers (see \cite[Definision 3.3.22]{V}), $I^m=I^{(m)}$ for all
$m>0$. Since any complete intersection complex is a matroid,
therefore by Theorem \ref{main}, $S/I^m$ is clean for all $m>0$.
\end{proof}

\begin{Example}\label{counter1}
{\em Let $I:=(x_1x_2,x_2x_3,x_3x_4)$. Obviously, $I$ ia an unmixed
square-free monomial ideal. Since $\mu(I)\leq 3$, it follows by
\cite[Corollary 2.6]{BDS} that $S/I$ is clean. On the other hand,
$I^\vee=(x_1x_3,x_2x_3,x_2x_4)$ is not matroidal. Hence, $I$ is not
the Stanley-Reisner ideal of a matroid. So by Corollary \ref{main2},
$S/I^{(m)}$ is not clean for all integer $m\geq 3$. Also, $S/I$ is
not complete intersection, so by Corollary \ref{c.i} $S/I^m$ is not
clean for all  integer $m\geq 3$. Now, consider the ideal $I$ as the
edge ideal of a graph $G$. Obviously, $G$ is a bipartite graph, so
by \cite[Corollary 10.3.17]{HH} $I$ is normally torsionfree.
Therefore for any $m$,
 $$\Ass(S/I^m)=\Ass(S/I)=\Min(I)=\Min(I^m).$$
  It follows by \cite[Corollary 3.5]{HP} that $S/I^m$ is not pretty clean for all  integer $m\geq 3$.}
\end{Example}

We note that the above example shows that, if $I\subset S$ is pretty clean monomial ideal, then necessarily
we do not have $S/I^{(m)}$ is pretty clean for all integer $m>0$.\\

\section{Second symbolic power and cleanness }
Let $\Delta$ be a 1-dimensional simplicial complex  and
$I=I_\Delta\subset S$. Minh and Trung in \cite{MT1} studied under
which conditions  $S/I^{(2)}$ and $S/I^2$ are Cohen-Macaulay. In
this section we will give a characterization for the
Cohen-Macaulayness of $S/I^{(2)}$ and $S/I^2$ in terms of the
cleanness property.

Let $G=(V,E)$ be a simple graph. In graph theory, the distance
between two vertices $u$ and $v$ of $G$ is the minimal length of
paths from $u$ to $v$ and is denoted by $\d(u,v)$. This length is
infinite if there is no path connecting them.  The diameter of $G$,
 $\diam(G)$, is defined by  $\diam(G):=\max\{\d(u,v)\;:\; u,v\in
V\}$.

 \begin{Theorem}\label{diam} Let $\Delta$ be a pure simplicial complex on $[n]$ with $\dim\Delta =1$ and $I=I_\Delta\subset S$.
  Then the following conditions are equivalent:
 \begin{enumerate}
\item[(a)] $S/I^{(2)}$   is clean.
\item[(b)] $S/I^{(2)}$   is Cohen-Macaulay.
\item[(c)]$\diam\Delta\leq 2$.
  \end{enumerate}
  \end{Theorem}

\begin{proof}
 $(a)\implies (b)$ Since $S/I^{(2)}$  is sequentially  Cohen-Macaulay and  $I^{(2)}$ is unmixed,  the desired conclusion follows from Lemma \ref{seq}.

 $(b)\implies (c)$ follows from
\cite[Theorem 2.3]{MT1}.

$(c)\implies (a)$  By \cite[Theorem 3.10]{S} it is enough to show
that $S/(I^{(2)})^p$ is clean. Let $I=I_\Delta=\cap_{i=1}^tP_{F_i}$
be a primary decomposition of $I$. Then $\Delta^c= \langle
F_1,\ldots,F_t\rangle$ with $|F_i|=n-2$ for all  $i=1,\ldots,t$. We
know that
\[
(I^{(2)})^p=\bigcap_{i=1}^t(P^2_{F_i})^p.
\]
If $F\subset [n]$, then by \cite[Proposition 2.5(2)]{F}
\[
(P^2_{F})^p=\underset{1\leq j\leq n-2}\bigcap P_{(F,2_j)}\cap
P_{(F,1)},
\]
where if $F=\{r_1,\ldots,r_{n-2}\}$ with $r_1<r_2<\cdots<r_{n-2}$,
then we set $(F,1):=\{(r_i,1) \;|\; r_i\in F\}$ and
$(F,2_j):=\{(r_j,2)\}\cup\{(r_i,1)\;|\; 1\leq i\leq n-2, i\neq j\}$.
Note that $(I^{(2)})^p$ is a monomial ideal in a polynomial ring
$T=K[x_{(1,1)},\ldots,x_{(n,1)},x_{(1,2)},\ldots, x_{(n,2)}]$. Since
$(I^{(2)})^p$ is the Stanley-Reisner ideal of simplicial complex
$$\Gamma=<(F_i,1)^c,(F_i,2_j)^c\;:\;1\leq i\leq t,\;1\leq j\leq
n-2>,$$  by a result of Dress \cite{D}, it is enough to prove that
$\Gamma$ is shellable.

We set $A_0:=\emptyset$ and
$A_i:=\{F_j^c\in\mathcal{F}(\Delta)\;|\;i\in
F_j^c\;\text{and}\;F_j^c\notin\bigcup_{s=1}^{i-1} A_{s}\}$ for all
$i=1,\ldots,n$. Note that $\mathcal{F}(\Delta)=\bigcup_{i=1}^n A_i$.
We order the facets of $\Gamma$ by the following process and show
that the given order is  a shelling order. For the convenience we
can assume that $A_1=\{F_1^c,\ldots,F_{s_1}^c\}$ for some $1\leq
s_1\leq t$. Let the initial part of our  order be as follow:
$$(*)\;(F_1,1)^c,(F_1,2_1)^c,\ldots,(F_1,2_{n-2})^c,(F_2,1)^c,(F_2,2_1)^c,\ldots,(F_2,2_{n-2})^c,\ldots,$$
$$(F_{s_1},1)^c,(F_{s_1},2_1)^c,\ldots,(F_{s_1},2_{n-2})^c.$$
Then  the following  inequalities hold:
$$n=|(F_1,1)^c\cap(F_1,2_{j})^c|-1=\dim(<(F_1,1)^c>\cap<(F_1,2_{j})^c>)$$
$$\leq\dim(<(F_1,1)^c,(F_1,2_1)^c,\ldots,(F_1,2_{j-1})^c>\cap<(F_1,2_{j})^c>)$$
$$\leq\dim<(F_1,2_{j})^c>-1=|(F_1,2_{j})^c|-2=n,$$
Now, let $2\leq d \leq s_1$. Then
$$n=\dim(<(F_1,1)^c>\cap<(F_d,1)^c>)$$
$$\leq\dim(<(F_1,1)^c,(F_1,2_1)^c,\ldots,(F_{d-1},2_{n-2})^c>\cap<(F_d,1)^c>)$$
$$\leq\dim<(F_d,1)^c>-1=n,$$
Also, for any $1\leq j\leq n-2$, we have
$$n=\dim(<(F_d,1)^c>\cap<(F_d,2_j)^c>)$$
$$\leq\dim(<(F_1,1)^c,(F_1,2_1)^c,\ldots,(F_d,1)^c,\ldots,(F_d,2_{j-1})^c>\cap<(F_d,2_j)^c>)$$
$$\leq\dim<(F_d,2_j)^c>-1=n.$$
Suppose that $\Gamma_1$ be a simplicial complex whose facets are all
of the sets belong to $(*)$. If we rename the facets of $\Gamma_1$
in the same order above by $G_1,\ldots,G_{s_1(n-1)}$, then it is
easy to see that  $<G_1,\ldots,G_{i-1}>\cap <G_i>$ is a pure
simplicial complex for all $i=1,\ldots,s_1(n-1)$. Therefore,
$\Gamma_1$ is a shellable.

Assume that $A_i=\{F_{s_{i-1}+1}^c,\ldots,F_{s_i}^c\}$ for $1\leq
i\leq h-1<n$, where $s_0=0$ and $s_{i-1}<s_i$.  Then we may assume by induction
process that the following order is a shelling order for the
simplicial complex with this set of facets.
$$(F_1,1)^c,(F_1,2_1)^c,\ldots,(F_1,2_{n-2})^c,\ldots,
(F_j,1)^c,(F_j,2_1)^c,\ldots,(F_j,2_{n-2})^c,$$
$$(F_{j+1},1)^c,(F_{j+1},2_1)^c,\ldots,(F_{j+1},2_{n-2})^c,\ldots,
(F_{s_{h-1}},1)^c,(F_{s_{h-1}},2_1)^c,\ldots,(F_{s_{h-1}},2_{n-2})^c,$$
where $1<j<s_{h-1}$.

Now, let $1<h\leq n$. If there exists $F^c\in\bigcup_{i=1}^{h-1}A_i$
such that $h\in F^c$, then we take an arbitrary element $G$ of $A_h$
and set $F_{s_{h-1}+1}^c:=G$.  In this case, we have
$$n=\dim(<(F,1)^c>\cap<(F_{s_{h-1}+1},1)^c>)$$
$$\leq\dim(<(F_1,1)^c,(F_1,2_1)^c,\ldots,(F_{s_{h-1}},2_{n-2})^c>\cap<(F_{s_{h-1}+1},1)^c>)$$
$$\leq\dim<(F_{s_{h-1}+1},1)^c)>-1=n.$$

Otherwise, for any  $F^c\in\bigcup_{i=1}^{h-1}A_i$,  $h\not\in F^c$.
Hence $\{1,h\}\not\in \mathcal F(\Delta)$. Since $\diam(\Delta)\leq
2$, it follows that there exists $m\in[n]$ such that $m\neq 1,h$ and
$\{m,h\}\in A_h$,  and  $F^c:=\{1,m\}\in A_1$. In this case we set
$F^c_{s_{h-1}+1}:=\{m,h\}$.

Now, the following inequalities hold:
$$n=\dim(<(F,1)^c>\cap<(F_{s_{h-1}+1},1)^c>)$$
$$\leq\dim(<(F_1,1)^c,(F_1,2_1)^c,\ldots,(F_{s_{h-1}},2_{n-2})^c>\cap<(F_{s_{h-1}+1},1)^c>)$$
$$\leq\dim<(F_{s_{h-1}+1},1)^c>-1=n.$$
We  order  all other facets  of $\Gamma$ which correspond to $A_h$,
as in the following:
$$(F_{s_{h-1}+1},2_1)^c,\ldots,(F_{s_{h-1}+1},2_{n-2})^c,\ldots,(F_{s_h},1)^c,(F_{s_h},2_1)^c,\ldots,(F_{s_h},2_{n-2})^c,$$
where $s_{h-1}<s_h$.

The same as previous, we can easily check that the given order is a
shelling order.
\end{proof}

A 1-dimensional simplicial complex $\Delta$ on the vertex set $[n]$
is called a cycle of length $n$ if  the facets of $\Delta$ are
$\{1,n\}$ and  $\{i,i+1\}$ for all $i=1,\ldots,n-1$.

\begin{Corollary} Let $\Delta$ be a pure simplicial complex on $[n]$
with $\dim\Delta =1$ and $I=I_\Delta\subset S$. Then the following
conditions are equivalent:
 \begin{enumerate}
\item[(a)] $S/I^2$   is clean.
\item[(b)] $S/I^2$   is Cohen-Macaulay.
\item[(c)]  $\Delta$ is a path of length $1,2$ or a cycle of length $3,4,5$.
  \end{enumerate}
  \end{Corollary}

 \begin{proof}
$(a)\implies (b)$ Since $S/I^2$  is sequentially  Cohen-Macaulay and
$I^2$ is unmixed,  the desired conclusion follows from Lemma
\ref{seq}.

  $(b)\implies (c)$ If $n=2$, then $\Delta$ is a path of length $1$.
If $n=3$, then $\Delta$ is either a path of length $2$ or a triangle
(a cycle of length $3$). Finally if $n\geq 4$, then by
\cite[Corollary 3.4]{MT1}, $\Delta$ is a cycle of length $4$ or $5$.

$(c)\implies (a)$ It is easy to see that in each case, we have
$\diam\Delta\leq 2$ and  $I^{(2)}=I^{2}$. Hence the desired
conclusion follows by Theorem \ref{diam}.
\end{proof}

It is known that if $I$ is a monomial ideal and $S/I$ is clean, then
$S/I$ is sequentially Cohen-Macaulay. In particular when $I$ is
unmixed, then $S/I$ is Cohen-Macaulay.  But the converse is not true
in general. In some special cases, like edge ideals of unmixed
bipartite graphs, it is known that Cohen-Macaulayness and cleanness
are equivalent.  As another corollary of our results we get the
following:

\begin{Corollary}
 Let $m>1$ be an integer,  $\Delta$  a pure simplicial complex with $\dim\Delta =1$, and $I=I_\Delta\subset S$. Then
 $S/I^{(m)}$ ($S/I^m$) is clean if and only if
 $S/I^{(m)}$ ($S/I^m$) is Cohen-Macaulay.
\end{Corollary}

\section*{Acknowledgments}
We would like to thank J\"urgen Herzog for raising the question of
whether all (symbolic) powers of a matroid  are clean and for
reading an earlier version of this manuscript.

\end{document}